\documentclass[
oribibl]{llncs}

\usepackage{amsmath,amsfonts,amssymb,latexsym,mathrsfs,stmaryrd}
\usepackage{subfigure}
\usepackage{tikz}
\usetikzlibrary{matrix,positioning}

\usepackage[all]{xy}
\UseComputerModernTips

\usepackage[all]{xy}
\newcommand{\hide}[1]{}

\newcommand{\ie}{\textit{i.e.}}
\newcommand{\eg}{\textit{e.g.}}
\newcommand{\confer}{\textit{cf.}}
\newcommand{\noemph}[1]{\emph{#1}}
\newcommand{\comp}{\circ}
\newcommand{\iso}{\cong}
\newcommand{\catequiv}{\simeq}
\newcommand{\id}{\mathrm{id}}
\newcommand{\tensor}{\otimes}
\newcommand{\juxt}{\oplus}
\newcommand{\FinDag}{\mathbf{Dag}}
\newcommand{\DagLetter}{D}
\newcommand{\DagPROP}{\prop\DagLetter}
\newcommand{\IdemObjPROP}{\prop V}
\newcommand{\EmptyPROP}{\prop P}
\newcommand{\LoopsPROP}[1]{\prop\LoopsLetter_{#1}}
\newcommand{\LoopPROP}{\LoopsPROP1}
\newcommand{\TreeLetter}{F}
\newcommand{\TreePROP}{\prop\TreeLetter}
\newcommand{\BiAlgLetter}{B}
\newcommand{\BiAlgPROP}{\prop\BiAlgLetter}
\newcommand{\BiAlgTheory}{\smTheory\BiAlgLetter}
\newcommand{\HopfAlgLetter}{H}
\newcommand{\HopfAlgPROP}{\prop\HopfAlgLetter}
\newcommand{\HopfAlgTheory}{\smTheory\HopfAlgLetter}
\newcommand{\DCBiAlgLetter}{R}
\newcommand{\DCBiAlgPROP}{\prop\DCBiAlgLetter}

\newcommand{\super}{\vee}
\newcommand{\setof}[1]{\{ #1 \}}
\newcommand{\bigsetof}[1]{\big\{ #1 \big\}}
\newcommand{\suchthat}{\mid}
\newcommand{\PROP}{PROP}
\newcommand{\PROPs}{\PROP{s}}
\newcommand{\Nat}{\mathbb{N}}
\newcommand{\Int}{\mathbb{Z}}
\newcommand{\FreeProp}[1]{\prop{P}[#1]}

\newcommand{\prop}[1]{\mathbf{#1}}
\newcommand{\FunMod}[2]{\mathsf{Mod}_{#1}(#2)}
\newcommand{\Alg}[2]{\mathsf{Alg}_{#1}(#2)}
\newcommand{\cat}[1]{\mathscr{#1}}
\newcommand{\g}{o}
\newcommand{\Sig}{\Sigma}
\newcommand{\Pres}{\mathcal T}
\newcommand{\Exp}{\mathrm E}
\newcommand{\Op}{O}
\newcommand{\op}{o}
\newcommand{\DCBiAlgLoopLetter}{D}
\newcommand{\DCBiAlgLoopTheory}{\smTheory\DCBiAlgLoopLetter}
\newcommand{\DagTheory}{\DCBiAlgLoopTheory}
\newcommand{\DCBiAlgTheory}{\smTheory\DCBiAlgLetter}
\newcommand{\LoopTheory}{\LoopsTheory 1}
\newcommand{\LoopsLetter}{N}
\newcommand{\LoopsTheory}[1]{\smTheory\LoopsLetter_{#1}}
\newcommand{\smTheory}[1]{\mathsf #1}
\newcommand{\symm}{\gamma}
\newcommand{\sem}[1]{\llbracket #1 \rrbracket}
\newcommand{\proj}{p}
\newcommand{\inj}{\imath}
\newcommand{\semD}[1]{\mathcal D\sem{#1}}
\newcommand{\psemD}[2]{\mathcal D_{#2}\sem{#1}}
\newcommand{\semR}[1]{\mathcal R\sem{#1}}
\newcommand{\semS}[1]{\mathcal N_1\sem{#1}}

\newcommand{\relu}{\cup}
\newcommand{\ord}[1]{\underline{#1}}
\newcommand{\ol}{\overline}
\newcommand{\after}{/}
\newcommand{\concat}{\comp}

\spnewtheorem*{example*}{Example}{\itshape}{\rmfamily}
 
\newcommand{\mysection}[1]{\section{#1}\setcounter{paragraph}{0}}
\newcommand{\myparagraph}[1]{\paragraph{\bf #1.}}

\begin{document}

\frontmatter

\title{The Algebra of Directed Acyclic Graphs}

\titlerunning{The Algebra of Directed Acyclic Graphs}

\toctitle{The Algebra of Directed Acyclic Graphs}


\author{Marcelo Fiore \and Marco Devesas Campos}

\authorrunning{Marcelo Fiore and Marco Devesas Campos}

\institute{Computer Laboratory\\University of Cambridge
}

\maketitle              

\begin{abstract}
We give an algebraic presentation of directed acyclic graph structure,
introducing a symmetric monoidal equational theory whose free {\PROP} we
characterise as that of finite abstract dags with input/output interfaces.
Our development provides an initial-algebra semantics for dag structure.
\keywords{dag, \PROP, symmetric monoidal equational theory, bialgebra,
  Hopf algebra, topological sorting, initial-algebra and categorical
  semantics}
\end{abstract}

\mysection{Introduction}

This work originated in a question of Robin Milner in connection to
explorations he was pursuing on possible extensions to his theory of
bigraphs~\cite{BigraphsBook}.
The particular direction that concerns us here is the generalisation of
the spatial dimension of bigraphs from a tree hierarchy to a directed
acyclic graph (dag) structure.

In~\cite{milner2005axioms}, Milner provided axioms for bigraphical
structure, axiomatising tree-branching structure by means of the
equational theory of commutative monoids.  As for the axiomatisation of
dag structure, he foresaw that it would also involve the dual theory of
commutative comonoids and, in conversation with the first author, raised
the question on how these two structures should interact.  In considering
the problem, it soon became clear that the axioms in question were those
of commutative bialgebras (where the monoid structure is a comonoid
homomorphism and, equivalently, the comonoid structure is a monoid
homomorphism) 
that are degenerate in that the composition of the comultiplication
followed by the multiplication collapses to the identity. 
%
This gives the axiomatics of wiring for dag structure.

The natural setting for presenting our work is the categorical language of
{\PROPs}; specifically relying on the concept of free {\PROP}, which
roughly corresponds to the symmetric strict monoidal category freely
generated by a symmetric monoidal equational theory.  Indeed, our main
result characterises the free {\PROP} on the theory~$\DagTheory$ of
degenerate commutative bialgebras with a node (endomap) as that of
finite abstract dags with input/output interfaces, see
Section~\ref{TheAlgebraOfIdags}.  Let us give an idea of why this is so.

It is important to note that the theory~$\DagTheory$ is the sum of two
sub-theories: the theory~$\DCBiAlgTheory$ of degenerate commutative
bialgebras and the theory~$\LoopTheory$ of a node (endomap).  Each of
these theories captures a different aspect of dag structure.  The free
{\PROP} on $\DCBiAlgTheory$ provides relational edge structure; while the
free {\PROP} on $\LoopTheory$ introduces node structure.  Thus, the free
{\PROP} on their sum, which is essentially obtained by interleaving both
structures, results in dag structure.  A main aim of the paper is to give
a simple technical development that formalises these intuitions.

This work falls within a central theme of Samson Abramsky's research:
the mathematical study of syntactic structure, an example of which
in the context of PROs is his characterisation of Temperley-Lieb
structure~\cite{Abramsky}.


\mysection{Directed acyclic graphs}

\myparagraph{Dags}

A \noemph{directed acyclic graph~(dag)} is a graph with directed edges in
which there are no cycles.  Formally, a directed graph is a pair~${(N, R
  \subseteq N \times N)}$ consisting of a set of nodes $N$ and a binary
relation~$R$ on it that specifies a directed edge from a node $n$ to
another one $m$ whenever $(n,m)\in R$.  The acyclicity condition of a
dag~$(N,R)$ is ensured by requiring that the transitive closure~$R^+$ of
the relation~$R$ is irreflexive; \ie~$(n,n) \notin R^+$ for all $n\in N$.

\myparagraph{Idags}

We will deal here with a slight generalisation of the notion of dag.  An
\noemph{interfaced dag (idag)} is a tuple of sets~$I,O,N$ and a 
binary
relation ${R\subseteq(I+N)\times(O+N)}$, for $+$ the 
sum 
of sets, subject to the acyclicity condition $(n,n) \notin (\proj \circ R
\circ \inj)^+$ for all $n\in N$, where the relations $i\subseteq
N\times(I+N)$ and $p\subseteq(O+N)\times N$ respectively denote the
injection of $N$ into $I+N$ and the projection of $O+N$ onto $N$.

Informally, idags are dags extended with interfaces.  An idag~$(I,O,N,R)$,
also referred to as an $(I,O)$-dag, is said to have input interface~$I$
and output interface~$O$; $N$ is its set of internal nodes.
Fig.~\ref{Adag} depicts two examples with input and output sets of
ordinals, where for $n\in\Nat$ we adopt the notation $\ord n$ for the
ordinal $\setof{0,\ldots,n-1}$.
\begin{figure}[h]
\centering
\subfigure{
\begin{tikzpicture}[>=stealth]
\draw (-3cm,1.5cm) node[fill,scale=0.4,shape=circle,label=left:$0$] (in0) {};
\draw (-3cm,0.5cm) node[fill,scale=0.4,shape=circle,label=left:$1$] (in1) {};

\draw (-1.5cm,0.5cm) node[draw,shape=circle,label=above:$k$] (k) {};
\draw (-1.5cm,0cm) node[draw,shape=circle,label=below:$l$] (l) {};
    
\draw (0cm,2cm) node[fill,scale=0.4,shape=circle,label=right:$0$] (out0) {};
\draw (0cm,1cm) node[fill,scale=0.4,shape=circle,label=right:$1$] (out1) {};
\draw (0cm,0cm) node[fill,scale=0.4,shape=circle,label=right:$2$] (out2) {};

\draw [->] (in0) -- (out1);
\draw [->] (in0) -- (k);
\draw [->] (in1) -- (k);
\draw [->] (k) -- (out1);
\draw [->] (k) -- (out2);
\draw [->] (l) -- (out2);
\end{tikzpicture}
}
\qquad\qquad
\subfigure{
\begin{tikzpicture}[>=stealth, baseline=-0.595cm]

\draw (1.5cm,1cm) node[draw,shape=circle,label=below:$a$] (a) {};
\draw (3cm,1.5cm) node[draw,shape=circle,label=below:$b$] (b) {};
\draw (1.5cm, 0cm) node[draw,shape=circle,label=below:$c$] (c) {};
\draw (3cm,0.5cm) node[draw,shape=circle,label=below:$d$] (d) {};

\draw (0cm,2cm) node[fill,scale=0.4,shape=circle,label=left:$0$] (in0) {};
\draw (0cm,1cm) node[fill,scale=0.4,shape=circle,label=left:$1$] (in1) {};
\draw (0cm,0cm) node[fill,scale=0.4,shape=circle,label=left:$2$] (in2) {};

\draw (4.5cm,1cm) node[fill,scale=0.4,shape=circle,label=right:$0$] (out1) {};

\draw [->] (in0) -- (a);
\draw [->] (in0) -- (b);
\draw [->] (in1) -- (a);
\draw [->] (in2) -- (a);
\draw [->] (in2) -- (c);
\draw [->] (a)   -- (b);
\draw [->] (a)   -- (d);
\draw [->] (b)   -- (out1);
\draw [->] (c)   -- (d);
\draw [->] (b)   -- (out1);
\draw [->] (d)   -- (out1);
\end{tikzpicture}
}
%
%
%
%
\caption{A $(\ord 2,\ord 3)$-dag and a $(\ord 3,\ord 1)$-dag.}
\label{Adag}
\end{figure}
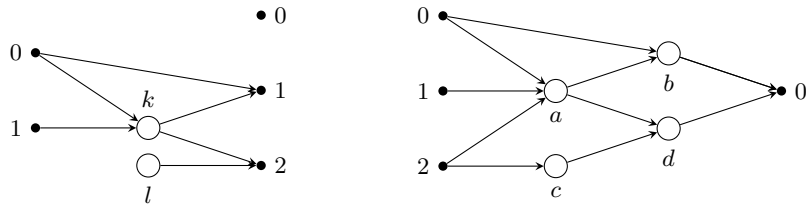

The notion of dag is recovered as that of idag with empty sets of input
and output nodes.  Idags also generalise binary relations, as these are in
bijective correspondence with idags without internal nodes.

\myparagraph{Operations on idags}

The extension of dags with interfaces allows for two basic operations on
them.

The \emph{concatenation} operation $D' \concat D$  of an \mbox{$(I,M)$-dag}
$D=(N,R)$ and an \mbox{$(M,O)$-dag} $D'=(N',R')$ is the \mbox{$(I,O)$-dag}
that retains the hierarchy information of both idags except that edges in
$R$ from input and internal nodes in $D$ to intermediate nodes in $M$
become redirected to target internal and output nodes in $D'$ as
specified by $R'$.
Formally, $D'\concat D =(N+N',R'')$ where $R''$ is 
the composite 
\[
  I+N+N'
  \xymatrix@C25pt{\ar[r]^-{R+\id}&}
  M+N+N'\iso M+N'+N
  \xymatrix@C30pt{\ar[r]^-{R'+\id}&}
  O+N'+N
  \iso
  O+N+N'
  \enspace.
\]

\begin{figure}[t]
\centering
\begin{tikzpicture}[>=stealth, yscale=1.5]
\draw (-1.5cm,1.5cm) node[fill,scale=0.4,shape=circle,label=left:$0$] (in0) {};
\draw (-1.5cm,0cm) node[fill,scale=0.4,shape=circle,label=left:$1$] (in1) {};

\draw (-0.5cm,0.75cm) node[draw,shape=circle,label=left:$k$] (k) {};
\draw (.5cm,0.75cm) node[draw,shape=circle,label=left:$l$] (l) {};
    
\draw (1.5cm,1.5cm) node[draw,shape=circle,label=below:$a$] (a) {};
\draw (3cm,1.5cm) node[draw,shape=circle,label=below:$b$] (b) {};
\draw (1.5cm, 0cm) node[draw,shape=circle,label=below:$c$] (c) {};
\draw (3cm,0cm) node[draw,shape=circle,label=below:$d$] (d) {};

\draw (4.5cm,0.75cm) node[fill,scale=0.4,shape=circle,label=right:$0$] (out1) {};

\draw [->] (in0) -- (a);
\draw [->] (in0) -- (k);
\draw [->] (in1) -- (k);
\draw [->] (k) -- (a);
\draw [->] (k) -- (c);
\draw [->] (l) -- (a);
\draw [->] (l) -- (c);
\draw [->] (a)   -- (b);
\draw [->] (a)   -- (d);
\draw [->] (b)   -- (out1);
\draw [->] (c)   -- (d);
\draw [->] (b)   -- (out1);
\draw [->] (d)   -- (out1);
\end{tikzpicture}
\caption{The concatenation of the idags of Fig.~\ref{Adag}.}
\end{figure}
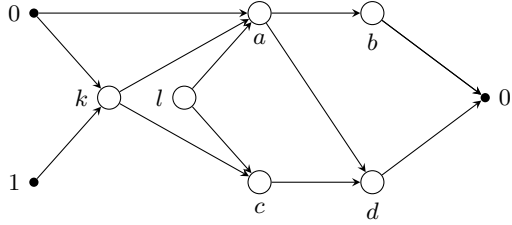

The \emph{juxtaposition} operation~$D\juxt D'$ of an \mbox{$(I,O)$-dag}
$D=(N,R)$ and an \mbox{$(I',O')$-dag} $D'=(N',R')$ is the
\mbox{$(I+I',O+O')$-dag} $D\juxt D'=(N+N', R'')$ where $R''$ is 
the composite 
%
%
\[
  I+I'+N+N'\iso I+N+I'+N'
  \xymatrix{\ar[r]^-{R+R'}&}
  O+N+O'+N'\iso O+O'+N+N'
  \enspace.
\]
Thus, juxtaposition puts two idags side by side, without modifying their
hierarchies. 


\hide{
The \emph{superposition} operation $D\super D'$ of two \mbox{$(I,O)$-dags}
$D=(N,R)$ and $D'=(N',R')$ is the \mbox{$(I,O)$-dag} $D\super D'= M_O
\comp (D\juxt D') \comp S_I$, where $S_I$ is the \mbox{$(I,I+I)$-dag}
$\big(\emptyset,\Delta_I+\id_\emptyset)$ and $M_O$ the
\mbox{$(O+O,O)$-dag} $(\emptyset,\nabla_O+\id_\emptyset)$.
}

\myparagraph{The category of finite abstract idags}

As we are to look at idags abstractly, we need a notion that identifies
those that are essentially the same.  Accordingly, we set two
\mbox{$(I,O)$-dags} $(N,R)$ and $(N',R')$ to be isomorphic whenever there
exists a bijection $\sigma:N \iso N'$ such that $(\id+\sigma)\comp R =
R'\comp (\id+\sigma)$.  

\noemph{Abstract \mbox{$(I,O)$-dags}} are then defined to be equivalence
classes of isomorphic \mbox{$(I,O)$-dags}.  The operations of
concatenation and juxtaposition respect isomorphism 
%
%
and one can use them to endow abstract idags with the structure of a
symmetric monoidal category.  We will restrict attention to the finite
case:  The category~$\FinDag$ has objects given by finite sets and
homs~$\FinDag(I,O)$ given by abstract \mbox{$(I,O)$-dags} with a finite
set of internal nodes.  These are equipped with composition operation
given by concatenation and identities given by identity relations.
Furthermore, the juxtaposition operation provides a symmetric tensor
product with unit the empty set.

The aim of the paper is to give an algebraic presentation characterising
$\FinDag$.  The appropriate setting for establishing our result is that of
{\PROP}s, to which we now turn.

\mysection{Product and permutation categories}

\myparagraph{\PROPs}

A \noemph{PROduct and Permutation category~(\PROP)},
see~\cite{mac1963categorical}, is a symmetric strict monoidal category
with objects the natural numbers and tensor product given by addition.
This definition is often relaxed in practice, allowing symmetric strict
monoidal categories with underlying commutative monoid structure on
objects isomorphic to the commutative monoid of natural numbers.  A
typical example is the additive monoid of finite
ordinals~$\big(\setof{\ord n\suchthat n\in\Nat }, \ord 0,\oplus\big)$, for
which $\ord n\oplus \ord m=\ord{n+m}$.

The main example of {\PROP} to be studied in the paper follows.
\begin{example*}
The category~$\FinDag$ is equivalent to the {\PROP}~$\DagPROP$ consisting of
its full subcategory determined by the finite ordinals.
\end{example*}

{\PROPs} describe algebraic structure, with the category $\FunMod{\prop
  P}{\cat C}$ of \emph{functorial models} of a {\PROP}~$\prop P$ in a
symmetric monoidal category~$\cat C$ given by symmetric monoidal
functors~$\prop P\to \cat C$ and symmetric monoidal natural
transformations between them.

\myparagraph{Free \PROPs} 

As remarked by Mac Lane in~\cite{mac1963categorical}, ``[a] useful
construction yields the free {\PROP} [\ldots] with given generators and
relations''; the usefulness residing in it being ``adapted to the study of
universal algebra''.  We briefly recall the construction and its universal
characterisation.

A \emph{signature} consists of a set of operators~$\Op$ together with an
assignment ${\Op\to\Nat\times\Nat}$ of arity/coarity pairs to operators.  In
this context, it is usual to use the notation ${\op:n\to m}$ to indicate that
the operator~$\op$ is assigned the arity/coarity pair~$(n,m)$.  
For a signature~$\Sig$, we let $\Exp(\Sig)$ consist of the expressions with
arity/coarity pairs generated by the language of symmetric strict monoidal
categories with underlying commutative monoid the additive natural numbers
together with the operators in $\Sig$ (\confer~\cite{Jay}).
A \emph{symmetric monoidal presentation} on a signature~$\Sig$ is then a set
of pairs of expressions in $\Exp(\Sig)$ with the same arity/coarity pair.
A \emph{symmetric monoidal equational theory} consists of a signature together
with a symmetric monoidal presentation on it.
An \emph{algebra} for a symmetric monoidal equational
theory~$(\Sig,\Pres)$ in a symmetric monoidal category is an object~$A$
equipped with morphisms $A^{\tensor n}\to A^{\tensor m}$ for every
operator of arity/coarity pair~$(n,m)$ in $\Sig$ such that the
interpretation of every equation in $\Pres$ is satisfied.  We write
$\Alg{(\Sig,\Pres)}{\cat C}$ for the category of
\mbox{$(\Sig,\Pres)$-algebras} and homomorphisms in $\cat C$.

The free {\PROP}~$\FreeProp{\Sig,\Pres}$ on a symmetric monoidal
theory~$(\Sig,\Pres)$ has homs $\FreeProp{\Sig,\Pres}(n,m)$ given by the
quotient of the set of expressions $\Exp(\Sig)$, with arity~$n$ and
coarity~$m$, under the laws of symmetric strict monoidal categories and
the presentation~$\Pres$.  It is universally characterised by a natural
equivalence 
\[
  \FunMod{\FreeProp{\Sig,\Pres}}{\cat C}
  \catequiv
  \Alg{(\Sig,\Pres)}{\cat C}
\]
for $\cat C$ ranging over symmetric monoidal categories.

\mysection{Examples of free {\PROPs}}

We give examples of symmetric monoidal equational theories together with
abstract characterisations of their induced free {\PROPs}.

\myparagraph{Empty theory}

The free {\PROP}~$\EmptyPROP$ on the \noemph{empty symmetric monoidal theory}
(with no operators and no equations) is the initial symmetric strict monoidal
category, \ie~the groupoid of finite ordinals and bijections.

\myparagraph{Nodes}

For a set~$L$, the free {\PROP}~$\LoopsPROP L$ on the symmetric monoidal
theory of \noemph{nodes}~$\LoopsTheory L = \big(\, \setof{ \lambda : 1\to
  1\,}_{\lambda \in L} , \emptyset \,\big)$ is the free symmetric strict
monoidal category on the free monoid~$(L^\star,\varepsilon,\cdot)$.  
Explicitly, $\LoopsPROP L$ has finite ordinals as objects and homs~$\LoopsPROP
L(\ord n,\ord m)= \EmptyPROP(\ord n,\ord m)\times (L^\star)^n$ with
identities~$(\id_n,\varepsilon)$
and
composition~${(\tau,w)\comp(\sigma,v) =
\big(\tau\comp\sigma,(w_{\sigma(i)}\cdot v_i)_{0\leq i<n}\big)}$.

\myparagraph{Idempotent objects}

The symmetric monoidal theory of an \noemph{idempotent object} has signature
with operators $\Delta:1\to 2$ and $\nabla:2\to 1$ subject to the presentation
\begin{gather*}
\nabla\comp\Delta = \id_1
\ : 1\to 1
\enspace,
\quad 
\Delta\comp\nabla = \id_2
\ : 2\to 2
\enspace.
\end{gather*}
The free {\PROP}~$\IdemObjPROP$ on it is a groupoid, with
hom~$\IdemObjPROP(1,1)$ given by Thompson's group~$V$,
see~\cite{FioreLeinster}.

\myparagraph{Commutative monoids and commutative comonoids}

The symmetric monoidal theory of \noemph{commutative monoids} has signature with
operators $\eta: 0 \rightarrow 1$ and $\nabla: 2 \rightarrow 1$ subject to the
presentation
\begin{gather*}
\nabla \comp (\eta \tensor \id_1) = \id_1 = \nabla \comp (\id_1 \tensor \eta)
\ : 1 \to 1
\enspace,\\
\nabla \comp (\nabla \tensor \id_1) = \nabla \comp (\id_1 \tensor \nabla)
\ : 3 \to 1
\enspace,\\
\nabla \comp \symm_{1,1} = \nabla
\ : 2 \to 2
\end{gather*}
where $\symm$ denotes the symmetry.
The free {\PROP} on it is the category of finite ordinals and
functions. 

The dual symmetric monoidal theory is that of \noemph{commutative comonoids}.
It has signature with operators $\epsilon: 1 \rightarrow 0$ and $\Delta: 1
\rightarrow 2$ subject to the presentation
\begin{gather*}
(\epsilon \tensor \id_1) \comp \Delta = \id_1 = (\id_1 \tensor \epsilon)
\comp \Delta
\ : 1\to 1
\enspace,
\\ 
(\Delta \tensor \id_1) \comp \Delta = (\id_1 \tensor \Delta)
\comp \Delta
\ : 1\to 3
\enspace,
\\
\symm_{1,1} \comp \Delta = \Delta
\ : 2\to 2
\enspace.
\end{gather*}
The free {\PROP} on it is of course the opposite of the category of finite
ordinals and functions.  

\myparagraph{Commutative bialgebras}
\label{CommutativeBialgebras}

The symmetric monoidal theory~$\BiAlgTheory$ of \noemph{commutative bialgebras}
has signature with operators $\eta: 0 \rightarrow 1$, $\nabla: 2 \rightarrow
1$, $\epsilon: 1 \rightarrow 0$, and $\Delta: 1 \rightarrow 2$ subject to the
presentation consisting of that of commutative monoids, commutative comonoids,
and the following
\begin{gather*}
\epsilon \comp \eta = \id_0
\ : 0\to 0
\enspace,
\\
\epsilon \comp \nabla = \epsilon \tensor \epsilon
\ : 2\to 0
\enspace,\quad
\Delta \comp \eta = \eta \tensor \eta
\ : 0\to 2
\enspace,\quad
\\
\Delta \comp \nabla 
= (\nabla \tensor \nabla) 
  \comp 
  (\id_1 \tensor \symm_{1,1} \tensor \id_1) 
  \comp
  (\Delta \tensor \Delta)
\ : 2 \to 2
\enspace.
\end{gather*}
The symmetric monoidal theory of \noemph{degenerate commutative bialgebras}
extends the above with the equation
\begin{gather*}
\nabla \comp \Delta = \id_1
\ : 1\to 1
\enspace.
\end{gather*}

The free {\PROP}~$\BiAlgPROP$ on the symmetric monoidal theory of commutative
bialgebras has homs~$\BiAlgPROP(n,m)=\Nat^{n\times m}$ under matrix
composition.  Accordingly, the free {\PROP}~$\DCBiAlgPROP$ on the symmetric
monoidal theory of degenerate commutative bialgebras is the category of finite
ordinals and relations.
See~\eg~\cite[\S10]{mac1963categorical},~\cite{pirashvili2002prop},
and~\cite{lack2004composing}.

\myparagraph{Commutative Hopf algebras}
\label{CommutativeHopfAlgebras}

The symmetric monoidal theory of \noemph{commutative Hopf algebras} extends that
of commutative bialgebras with an \noemph{antipode} operator~$s: 1\to 1$ subject
to the laws: 
\[\begin{array}{c}
\begin{array}{cc}
s \comp \eta = \eta 
\ : 0\to1
\enspace,
&
\nabla \comp (s \oplus s) = s \circ \nabla
\ : 2\to1
\enspace,
\\[1mm]
\epsilon \comp s = \epsilon
\ : 1\to0
\enspace,
&
(s \oplus s) \comp \Delta = \Delta \comp s 
\ : 1\to2
\enspace,
\end{array}
\\[4mm]
\nabla \comp (s \oplus \id_1) \comp \Delta 
= \eta \comp \epsilon 
= \nabla \comp (\id_1 \oplus s) \comp \Delta
\ : 1\to1
\enspace.
\end{array}\]
Its free {\PROP}~$\HopfAlgPROP$ has homs~$\HopfAlgPROP(n,m)=\Int^{n\times m}$
under matrix composition.\footnote{We are grateful to Ross Duncan and
  Aleks Kissinger for bringing this example to our attention.}

\myparagraph{Commutative monoids with a node}

The symmetric monoidal theory of \noemph{commutative monoids with a node} is
the sum of the theory of commutative monoids and the theory of a single
node.  Its free {\PROP}~$\TreePROP$ has homs consisting of interfaced
forests.  Precisely, $\TreePROP$ is the sub-{\PROP} of $\DagPROP$
determined by the interfaced dags~$(N,R)$ with $R$ a total function,
see~\cite{milner2005axioms}.  

\mysection{The algebra of idags}
\label{TheAlgebraOfIdags}

\myparagraph{Algebraic structure} 

The generator $\ord 1$ of the {\PROP}~$\DagPROP$ carries two important
algebraic structures: 
\newpage
\begin{enumerate}
\item the degenerate commutative bialgebra
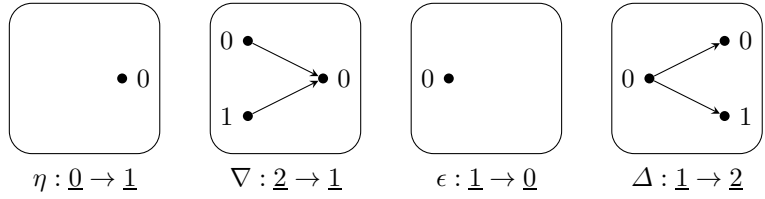
\begin{center}
\begin{minipage}[b]{0.22\linewidth}
\centering
\begin{tikzpicture}[>=stealth]
\draw[rounded corners=0.3cm] (-0.5cm,-0.5cm) rectangle (1.5cm,1.5cm);

\draw (1cm,0.5cm) node[fill,scale=0.4,shape=circle,label=right:0] (out0) {};
\end{tikzpicture}
\\
$\eta: \ord 0 \rightarrow \ord 1$
\end{minipage}
\begin{minipage}[b]{0.22\linewidth}
\centering
\begin{tikzpicture}[>=stealth]
\draw[rounded corners=0.3cm] (-0.5cm,-0.5cm) rectangle (1.5cm,1.5cm);

\draw (0cm,1cm) node[fill,scale=0.4,shape=circle,label=left:0] (in0) {};
\draw (0cm,0cm) node[fill,scale=0.4,shape=circle,label=left:1] (in1) {};
\draw (1cm,0.5cm) node[fill,scale=0.4,shape=circle,label=right:0] (out0) {};

\draw [->] (in0) -- (out0);
\draw [->] (in1) -- (out0);
\end{tikzpicture}
\\
$\nabla:\ord 2 \rightarrow \ord 1$
\end{minipage}
\begin{minipage}[b]{0.22\linewidth}
\centering
\begin{tikzpicture}[>=stealth]
\draw[rounded corners=0.3cm] (-0.5cm,-0.5cm) rectangle (1.5cm,1.5cm);

\draw (0cm,0.5cm) node[fill,scale=0.4,shape=circle,label=left:0] (in0) {};
\end{tikzpicture}
\\
$\epsilon:\ord 1 \rightarrow \ord 0$
\end{minipage}
\begin{minipage}[b]{0.22\linewidth}
\centering
\begin{tikzpicture}[>=stealth]
\draw[rounded corners=0.3cm] (-0.5cm,-0.5cm) rectangle (1.5cm,1.5cm);

\draw (0cm,0.5cm) node[fill,scale=0.4,shape=circle,label=left:0] (in0) {};
\draw (1cm,1cm) node[fill,scale=0.4,shape=circle,label=right:0] (out0) {};
\draw (1cm,0cm) node[fill,scale=0.4,shape=circle,label=right:1] (out1) {};

\draw [->] (in0) -- (out1);
\draw [->] (in0) -- (out0);
\end{tikzpicture}
\\
$\Delta:\ord 1 \rightarrow \ord 2$
\end{minipage}
\\
\vspace{7pt}
\end{center}
and

\item the node
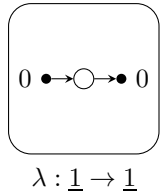
\begin{center}\begin{minipage}[b]{0.2\linewidth}
\centering
\begin{tikzpicture}[>=stealth]
\draw[rounded corners=0.3cm] (-0.5cm,-0.5cm) rectangle (1.5cm,1.5cm);

\draw (0cm,0.5cm) node[fill,scale=0.4,shape=circle,label=left:0] (in0) {};
\draw (0.5cm,0.5cm) node[draw, scale=0.8,shape=circle 
] (g) {};
\draw (1cm,0.5cm) node[fill,scale=0.4,shape=circle,label=right:0] (out0) {};

\draw [->] (in0) -- (g);
\draw [->] (g) -- (out0);
\end{tikzpicture}
\\
$\lambda:\ord 1 \rightarrow \ord 1$
\end{minipage}\end{center}
\end{enumerate}
These respectively induce universal injections of {\PROPs} as follows 
\begin{equation}\label{Cospan}
  \begin{array}{c}\xymatrix@C15pt@R15pt{
 \DCBiAlgPROP \ar[dr] & & \ar[dl]\LoopPROP
   \\
   & \DagPROP & 
  }\end{array}
\end{equation}
The main result of the paper is that together the {\PROPs}~$\DCBiAlgPROP$ and
$\LoopPROP$ characterise the {\PROP}~$\DagPROP$.
\begin{theorem}
  The {\PROP}~$\DagPROP$ is free on the symmetric monoidal theory
  $\DCBiAlgLoopTheory$ of \emph{degenerate commutative bialgebras with a
    node} (\ie~the sum of the theory~$\DCBiAlgTheory$ of degenerate
  commutative bialgebras and the theory~$\LoopTheory$ of a node).
\end{theorem}

\hide{
$\DCBiAlgLoopTheory$ contains no equations involving the node.  As a
consequence the free {\PROP} is not additive; in particular, it does not
respect distributivity, as~\eg~$\nabla \comp (\g \tensor \g) \comp \Delta
= \g \comp \nabla \comp \Delta$ is not provable.  Multiplication and
comultiplication are no longer natural transformations. 
}

The theorem is proved by establishing the universal property of the free
{\PROP} by means of the following lemma, whose proof occupies the rest of
the section.
\begin{lemma}\label{PushoutLemma}
  The cospan~(\ref{Cospan}) is a pushout of symmetric monoidal categories for
  the following span of universal {\PROP} injections 
  \[\xymatrix@C15pt@R15pt{
      & \ar[dl] \EmptyPROP \ar[dr] & 
      \\
      \DCBiAlgPROP & & \LoopPROP
    }\]
\end{lemma}

\myparagraph{Categorical interpretation}

We start by giving an interpretation of finite idags on degenerate
commutative bialgebras with a node in arbitrary symmetric monoidal
categories.  Specifically, for every $\DCBiAlgLoopTheory$-algebra 
\[ 
  (A,\eta_A:I\to A,\nabla_A:A\tensor A\to A,\epsilon_A:A\to
  I,\Delta_A:A\to A\tensor A, \lambda_A:A\to A)
\]
in a symmetric monoidal category~$\cat C$ we will define mappings
\[
  \semD-_A
  : \DagPROP(\ord n,\ord m) \to \cat C(A^{\tensor n},A^{\tensor m})
\]
extending the interpretations for $\DCBiAlgPROP$ and $\LoopPROP$,
respectively induced by the \mbox{$\DCBiAlgTheory$-algebra}
$(A,\eta_A,\nabla_A,\epsilon_A,\Delta_A)$ and the \mbox{$\LoopTheory$-algebra}
$(A,\lambda_A)$, as follows 
\[\xymatrix{
    \ar[dr]_-{\semR-_A} \DCBiAlgPROP(\ord n,\ord m) \ar[r] &
    \DagPROP(\ord n,\ord m) \ar[d] & \ar[l] \LoopPROP(\ord n,\ord m)
    \ar[dl]^-{\semS-_A} \\
    & \cat C(A^{\tensor n},A^{\tensor m}) & 
  }\]

For dag structure, in stark contrast with tree structure, there is no
direct definition of the interpretation function by structural induction,
and a more involved approach to defining it is necessary.  This proceeds
in two steps as follows.
\begin{enumerate}
  \item
    We give an interpretation $\psemD D\sigma_A$ parameterised by
    topological sortings~$\sigma$ of $D$.
  \item
    We show that the interpretation is independent of the topological
    sorting, in that $\psemD D\sigma _A = \psemD D{\sigma'}_A$ for all
    topological sortings~$\sigma$ and $\sigma'$ of $D$.  
\end{enumerate}

A \emph{topological sorting} of a finite 
\mbox{$(\ord n,\ord m)$-dag}~$D=(N,R)$ is a bijection 
\[
  \sigma: [N]\to N
  \enspace,\quad
  \mbox{for $[N]=\bigsetof{\, 0,\ldots,\mbox{$|N|\!-\!1$} \,}$}
\]
such that 
\[
  \forall\: 0\leq i,j <|N|.\, 
  \big(\inj_2(\sigma_i),\inj_2(\sigma_j)\big) 
    \in R 
  \implies 
  i < j 
\]
where $\inj_2$ denotes the second sum injection. 
Every such topological sorting induces a canonical decomposition in $\DagPROP$
as follows: 
\begin{equation}\label{DagDecomposition}
    D 
    = 
    D_{|N|}^\sigma \comp (\id_{n+|N|-1} \oplus \lambda) \comp
    D_{|N|-1}^\sigma \comp \cdots \comp (\id_{n} \oplus \lambda) \comp
    D_0^\sigma
\end{equation}
where, 
for $0\leq k<|N|$, each
$D_k^\sigma\in\DagPROP\big(\ord n\oplus\ord k,\ord n\oplus\ord
k\oplus\ord1\big)$ corresponds to the relation $R_k^\sigma =
(\inj_{n+k}\relu \ol R_k^\sigma)\in \DCBiAlgPROP\big(\ord n\oplus \ord
k,\ord n\oplus \ord k\oplus \ord 1\big)$ with $\inj_{n+k}$ the inclusion
relation and $\ol R_k^\sigma$ encoding the edges from the input nodes
$0\leq i<n$ and the internal nodes $\sigma_\ell$ for $0\leq \ell<k$ to the
internal node $\sigma_k$; 
while $D_{|N|}^\sigma\in\DagPROP\big(\ord n\oplus[N],\ord m\big)$ corresponds
to the relation 
$R_{|N|}^\sigma\in\DCBiAlgPROP\big(\ord n\oplus[N],\ord m\big)$ encoding the
edges from the input and the internal nodes to the output nodes.
Explicitly, 
for $0\leq k<|N|$ and $0\leq j<m$, 
\begin{itemize}
\item
$\forall\: 0\leq i < n .\,
  (i,n+k) \in \ol R_k^\sigma
  \iff  
  \big(\inj_1(i), \inj_2(\sigma_k)\big) \in R$
\smallskip

\item
$\forall\: 0\leq \ell < k.\,
  (n+\ell,n+k) \in \ol R_k^\sigma
  \iff  
  \big(\inj_2 (\sigma_\ell), \inj_2(\sigma_k)\big) \in R$
\smallskip

\item
$\forall\: 0\leq i < n.\,
  (i,j) \in R_{|N|}^\sigma 
  \iff 
  \big(\inj_1(i),\inj_1(j)\big) \in R$
\smallskip

\item
$\forall\: 0\leq \ell < |N|.\,
  (n+\ell, j) \in R_{|N|}^\sigma
  \iff 
  \big(\inj_2 (\sigma_\ell),\inj_1(j)\big) \in R$
\end{itemize}
where $\inj_1$ and $\inj_2$ respectively denote the first and second sum
injections. 
See Fig.~\ref{fig:dag_decomposition} for two sample decompositions.

\begin{figure}[t]
\centering
\hide{
\subfigure{
\begin{tikzpicture}[>=stealth]

\draw (1.5cm,1cm) node[draw,shape=circle,label=below:$a$] (a) {};
\draw (3cm,2cm) node[draw,shape=circle,label=below:$b$] (b) {};
\draw (1.5cm, 0cm) node[draw,shape=circle,label=below:$c$] (c) {};
\draw (4.5cm,1cm) node[draw,shape=circle,label=below:$d$] (d) {};

\draw (0cm,2cm) node[fill,scale=0.4,shape=circle,label=left:0] (in0) {};
\draw (0cm,1cm) node[fill,scale=0.4,shape=circle,label=left:1] (in1) {};
\draw (0cm,0cm) node[fill,scale=0.4,shape=circle,label=left:2] (in2) {};

\draw (6cm,1cm) node[fill,scale=0.4,shape=circle,label=right:0] (out1) {};

\draw [->] (in0) -- (a);
\draw [->] (in0) -- (b);
\draw [->] (in1) -- (a);
\draw [->] (in2) -- (c);
\draw [->] (a)   -- (b);
\draw [->] (b)   -- (d);
\draw [->] (b)   -- (out1);
\draw [->] (c)   -- (d);
\draw [->] (d)   -- (out1);
\end{tikzpicture}
}
}
\subfigure{
\begin{tikzpicture}[>=stealth,scale=0.8]

\draw (0cm,2.5cm) node[fill,scale=0.4,shape=circle,label=left:0] (in0) {};
\draw (0cm,2cm) node[fill,scale=0.4,shape=circle,label=left:1] (in1) {};
\draw (0cm,1.5cm) node[fill,scale=0.4,shape=circle,label=left:2] (in2) {};

\draw (0cm,2.5cm) -- (9cm,2.5cm);
\draw (9cm,2.4cm) -- ++(0cm,0.2cm);
\draw (0cm,2cm)   -- (9cm,2cm);
\draw (9cm,1.9cm) -- ++(0cm,0.2cm);
\draw (0cm,1.5cm) -- (9cm,1.5cm);
\draw (9cm,1.4cm) -- ++(0cm,0.2cm);

\draw (1cm,2.5cm) node[fill,scale=0.4,shape=circle] (ain0) {};
\draw (1cm,2cm)   node[fill,scale=0.4,shape=circle] (ain1) {};
\draw (1cm,1.5cm)   node[fill,scale=0.4,shape=circle] (ain2) {};
\draw (1.5cm,1cm) node[fill,scale=0.4,shape=circle] (aout) {};
\draw (2cm,1cm)   node[draw,shape=circle,label=below:$a$] (a) {};
\draw (ain0) -- (aout);
\draw (ain1) -- (aout);
\draw (ain2) -- (aout);
\draw (aout) -- (a);
\draw (a) -- (9cm,1cm);
\draw (9cm,0.9cm) -- ++(0cm,0.2cm);
\draw[color=gray,dashed] (0.75,2.75cm) -- ++(1cm,0cm) -- ++(0cm,-2cm) -- ++(-1cm,0) -- cycle;
\draw (1.25cm,2.75cm) node[label=above:$D_0^\sigma$] () {};

\draw (3cm,2.5cm)   node[fill,scale=0.4,shape=circle] (bin0) {};
\draw (3cm,1cm)     node[fill,scale=0.4,shape=circle] (bin1) {};
\draw (3.5cm,0.5cm) node[fill,scale=0.4,shape=circle] (bout) {};
\draw (4cm,0.5cm)   node[draw,shape=circle,label=below:$b$] (b) {};
\draw (bin0) -- (bout);
\draw (bin1) -- (bout);
\draw (bout) -- (b);
\draw (b) -- (9cm,0.5cm);
\draw[color=gray,dashed] (2.75cm,2.75) -- ++(1cm,0cm) -- ++(0cm,-2.5cm) -- ++(-1cm,0) -- cycle;
\draw (3.25cm,2.75cm) node[label=above:$D_1^\sigma$] () {};

\draw (5cm,1.5cm) node[fill,scale=0.4,shape=circle] (cin0) {};
\draw (5.5cm,0cm) node[fill,scale=0.4,shape=circle] (cout) {};
\draw (6cm,0cm) node[draw,shape=circle,label=below:$c$] (c) {};
\draw (cin0) -- (cout);
\draw (cout) -- (c);
\draw (c) -- (9cm,0cm);
\draw (9cm,-0.1cm) -- ++(0cm,0.2cm);
\draw[color=gray,dashed] (4.75,2.75) -- ++(1cm,0cm) -- ++(0cm,-3cm) -- ++(-1cm,0) -- cycle;
\draw (5.25cm,2.75cm) node[label=above:$D_2^\sigma$] () {};

\draw (7cm,1cm) node[fill,scale=0.4,shape=circle] (din0) {};
\draw (7cm,0cm) node[fill,scale=0.4,shape=circle] (din1) {};
\draw (7.5cm,-0.5cm) node[fill,scale=0.4,shape=circle] (dout) {};
\draw (8cm,-0.5cm) node[draw,shape=circle,label=below:$d$] (d) {};
\draw (din0) -- (dout);
\draw (din1) -- (dout);
\draw (dout) -- (d);
\draw[color=gray,dashed] (6.75,2.75) -- ++(1cm,0cm) -- ++(0cm,-3.5cm) -- ++(-1cm,0) -- cycle;
\draw (7.25cm,2.75cm) node[label=above:$D_3^\sigma$] () {};

\draw[color=gray,dashed] (8.75cm,2.75cm) -- ++(1cm,0cm) -- ++(0cm,-3.5cm) -- ++(-1cm,0) -- cycle;
\draw (9.25cm,2.75cm) node[label=above:$D_4^\sigma$] () {};

\draw (9cm,0.5cm) node[fill,scale=0.4,shape=circle] (outb) {};
\draw (9.5cm,-0.5cm) node[fill,scale=0.4,shape=circle] (out1) {};
\draw (10cm,-0.5cm) node[fill,scale=0.4,shape=circle,label=right:0] (out) {};

\draw (9.0cm,-0.5cm) node[fill,scale=0.4,shape=circle] (newdin) {};
\draw (outb) -- (out1);
\draw (d) -- (out1);
\draw (out1) -- (out);

\end{tikzpicture}
}
\subfigure{
\begin{tikzpicture}[>=stealth,scale=0.8]

\draw (0cm,2.5cm) node[fill,scale=0.4,shape=circle,label=left:$0$] (in0) {};
\draw (0cm,2cm) node[fill,scale=0.4,shape=circle,label=left:$1$] (in1) {};
\draw (0cm,1.5cm) node[fill,scale=0.4,shape=circle,label=left:$2$] (in2) {};

\draw (0cm,2.5cm) -- (9cm,2.5cm);
\draw (9cm,2.4cm) -- ++(0cm,0.2cm);
\draw (0cm,2cm)   -- (9cm,2cm);
\draw (9cm,1.9cm) -- ++(0cm,0.2cm);
\draw (0cm,1.5cm) -- (9cm,1.5cm);
\draw (9cm,1.4cm) -- ++(0cm,0.2cm);

\draw (1cm,2.5cm) node[fill,scale=0.4,shape=circle] (ain0) {};
\draw (1cm,2cm)   node[fill,scale=0.4,shape=circle] (ain1) {};
\draw (1cm,1.5cm)   node[fill,scale=0.4,shape=circle] (ain2) {};
\draw (1.5cm,1cm) node[fill,scale=0.4,shape=circle] (aout) {};
\draw (2cm,1cm)   node[draw,shape=circle,label=below:$a$] (a) {};
\draw (ain0) -- (aout);
\draw (ain1) -- (aout);
\draw (ain2) -- (aout);
\draw (aout) -- (a);
\draw (a) -- (9cm,1cm);
\draw (9cm,0.9cm) -- ++(0cm,0.2cm);
\draw[color=gray,dashed] (0.75,2.75cm) -- ++(1cm,0cm) -- ++(0cm,-2cm) -- ++(-1cm,0) -- cycle;
\draw (1.25cm,2.75cm) node[label=above:$D_0^{\sigma'}$] () {};

\draw (3cm,1.5cm) node[fill,scale=0.4,shape=circle] (cin0) {};
\draw (3.5cm,0.5cm) node[fill,scale=0.4,shape=circle] (cout) {};
\draw (4cm,0.5cm)   node[draw,shape=circle,label=below:$c$] (c) {};
\draw (cin0) -- (cout);
\draw (cout) -- (c);
\draw (c) -- (9cm,0.5cm);
\draw (9cm,0.4cm) -- ++(0cm,0.2cm);
\draw[color=gray,dashed] (2.75cm,2.75) -- ++(1cm,0cm) -- ++(0cm,-2.5cm) -- ++(-1cm,0) -- cycle;
\draw (3.25cm,2.75cm) node[label=above:$D_1^{\sigma'}$] () {};

\draw (5cm,2.5cm)   node[fill,scale=0.4,shape=circle] (bin0) {};
\draw (5cm,1cm)     node[fill,scale=0.4,shape=circle] (bin1) {};
\draw (5.5cm,0cm) node[fill,scale=0.4,shape=circle] (bout) {};
\draw (9cm,0cm) node[fill,scale=0.4,shape=circle] (newb) {};
\draw (6cm,0cm) node[draw,shape=circle,label=below:$b$] (b) {};
\draw (bin0) -- (bout);
\draw (bin1) -- (bout);
\draw (bout) -- (b);
\draw (b) -- (9cm,0cm);
\draw[color=gray,dashed] (4.75,2.75) -- ++(1cm,0cm) -- ++(0cm,-3cm) -- ++(-1cm,0) -- cycle;
\draw (5.25cm,2.75cm) node[label=above:$D_2^{\sigma'}$] () {};

\draw (7cm,0.5cm) node[fill,scale=0.4,shape=circle] (din0) {};
\draw (7cm,1cm) node[fill,scale=0.4,shape=circle] (din1) {};
\draw (7.5cm,-0.5cm) node[fill,scale=0.4,shape=circle] (dout) {};
\draw (8cm,-0.5cm) node[draw,shape=circle,label=below:$d$] (d) {};
\draw (din0) -- (dout);
\draw (din1) -- (dout);
\draw (dout) -- (d);
\draw[color=gray,dashed] (6.75,2.75) -- ++(1cm,0cm) -- ++(0cm,-3.5cm) -- ++(-1cm,0) -- cycle;
\draw (7.25cm,2.75cm) node[label=above:$D_3^{\sigma'}$] () {};

\draw[color=gray,dashed] (8.75cm,2.75cm) -- ++(1cm,0cm) -- ++(0cm,-3.5cm) -- ++(-1cm,0) -- cycle;
\draw (9.25cm,2.75cm) node[label=above:$D_4^{\sigma'}$] () {};

\draw (9.5cm,-0.5cm) node[fill,scale=0.4,shape=circle] (out1) {};
\draw (9.0cm,-0.5cm) node[fill,scale=0.4,shape=circle] (newdin) {};
\draw (10cm,-0.5cm) node[fill,scale=0.4,shape=circle,label=right:$0$] (out) {};
\draw (d) -- (out1);
\draw (out1) -- (out);
\draw (newb) -- (out1);

\end{tikzpicture}
}
\\
\subfigure{
\raisebox{-.3\height}{\begin{tikzpicture}[>=stealth,scale=0.8]
\draw (2.5cm,2.5cm) -- (4cm,2.5cm);
\draw (2.5cm,2cm) -- (4cm,2cm);
\draw (2.5cm,1.5cm) -- (4cm,1.5cm);
\draw (2.5cm,1cm) -- (4cm,1cm);

\draw (3cm,2.5cm)   node[fill,scale=0.4,shape=circle] (bin0) {};
\draw (3cm,1cm)     node[fill,scale=0.4,shape=circle] (bin1) {};
\draw (3.5cm,0.5cm) node[fill,scale=0.4,shape=circle] (bout) {};
\draw (bin0) -- (bout);
\draw (bin1) -- (bout);
\draw (bout) -- ++(0.5cm,0);
\draw[color=gray,dashed] (2.75cm,2.75) -- ++(1cm,0cm) -- ++(0cm,-2.5cm) -- ++(-1cm,0) -- cycle;
\draw (3.25cm,2.75cm) node[label=above:$R_1^\sigma$] () {};

\end{tikzpicture}}
\quad = \enspace
\raisebox{-.3\height}{\begin{tikzpicture}[>=stealth,scale=0.8]
\draw (2.5cm,2.5cm) -- (4cm,2.5cm);
\draw (2.5cm,2cm) -- (4cm,2cm);
\draw (2.5cm,1.5cm) -- (4cm,1.5cm);
\draw (2.5cm,1cm) -- (4cm,1cm);
\draw (3.5cm,0.5cm) -- (4cm,0.5cm);
\draw (3.5cm,0.625cm)-- (3.5cm,0.375cm);

\draw[color=gray,dashed] (2.75cm,2.75) -- ++(1cm,0cm) -- ++(0cm,-2.5cm) -- ++(-1cm,0) -- cycle;
\draw (3.25cm,2.75cm) node[label=above:$\inj_4$] () {};

\end{tikzpicture}}
\enspace $\relu$ \
\raisebox{-.3\height}{\begin{tikzpicture}[>=stealth,scale=0.8]
\draw (3cm,2.5cm)   node[fill,scale=0.4,shape=circle] (bin0) {};
\draw (3cm,1cm)     node[fill,scale=0.4,shape=circle] (bin1) {};
\draw (3.5cm,0.5cm) node[fill,scale=0.4,shape=circle] (bout) {};

\draw (2.5cm,2.5cm) -- (bin0);
\draw (2.5cm,2cm) -- (3cm,2cm);
\draw (3cm,2.125cm) -- (3cm,1.875cm);
\draw (2.5cm,1.5cm) -- (3cm,1.5cm);
\draw (3cm,1.625cm) -- (3cm,1.375cm);
\draw (2.5cm,1cm) -- (bin1);

\draw (bin0) -- (bout);
\draw (bin1) -- (bout);
\draw (bout) -- ++(0.5cm,0);

\draw (3.5cm,2.5cm) -- (4cm,2.5cm);
\draw (3.5cm,2.625cm) -- (3.5cm,2.375cm);
\draw (3.5cm,2cm) -- (4cm,2cm);
\draw (3.5cm,2.125cm) -- (3.5cm,1.875cm);
\draw (3.5cm,1.5cm) -- (4cm,1.5cm);
\draw (3.5cm,1.625cm) -- (3.5cm,1.375cm);
\draw (3.5cm,1cm) -- (4cm,1cm);
\draw (3.5cm,1.125cm) -- (3.5cm,0.875cm);

\draw[color=gray,dashed] (2.75cm,2.75) -- ++(1cm,0cm) -- ++(0cm,-2.5cm) -- ++(-1cm,0) -- cycle;
\draw (3.25cm,2.75cm) node[label=above:$\ol R_1^\sigma$] () {};

\end{tikzpicture}}
}
\caption{Decompositions of the (\ord3,\ord1)-dag of Fig.~\ref{Adag} for
  the topological sortings~$\sigma=(a,b,c,d)$ and $\sigma'=(a,c,b,d)$; and
  the decomposition of an auxiliary relation.}
\label{fig:dag_decomposition}
\end{figure}
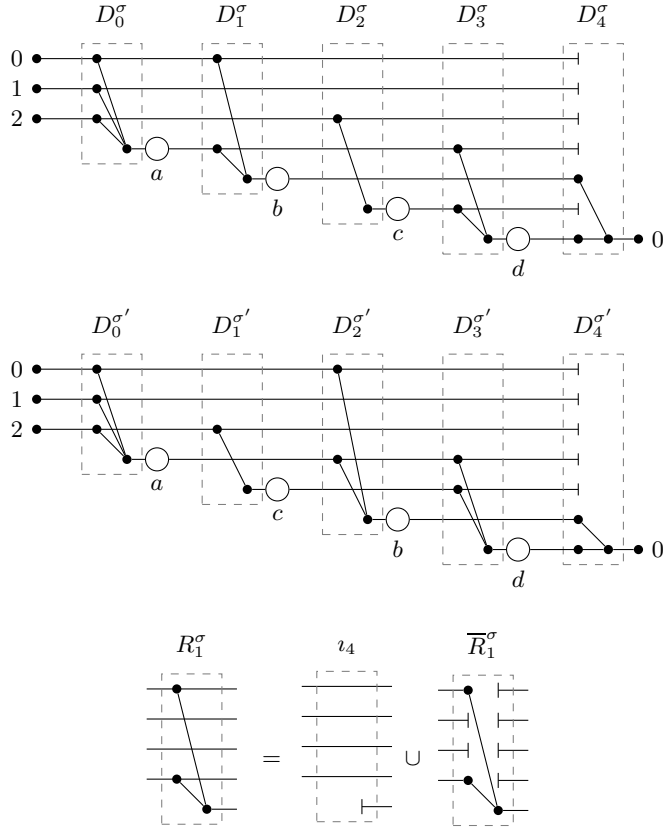

For a finite $(\ord n,\ord m)$-dag~$D$, we are led to define $\psemD
D\sigma_A: A^{\tensor n}\to A^{\tensor m}$ as the composite
\[ 
  \semR{R_{|N|}^\sigma}_A \comp (\id_{n+|N|-1} \tensor \lambda_A) \comp
  \semR{R_{|N|-1}^\sigma}_A \comp \cdots \comp (\id_{n} \tensor
  \lambda_A) \comp \semR{ R_0^\sigma }_A
  \enspace.
\]

The above definitions have been specifically chosen so that the properties
to follow are readily established.

\bigskip

A first remark is that the interpretation is invariant under isomorphism.
\begin{proposition}
Let $D=(N,R)$ and $D'=(N',R')$ be two finite $(\ord n,\ord m)$-dags
isomorphic by means of a bijection $\beta:N\iso N'$.  If $\sigma$ is a
topological sorting of $D$, then $\sigma'=\beta\comp\sigma$ is a
topological sorting of $D'$ and $\psemD D\sigma_A =\psemD{D'}{\sigma'}_A$.
\end{proposition}
\begin{proof}
Because one has by construction that $R_i^\sigma={R'}^{\sigma'}_i$ for all
$0\leq i\leq|N|$.  
\end{proof}

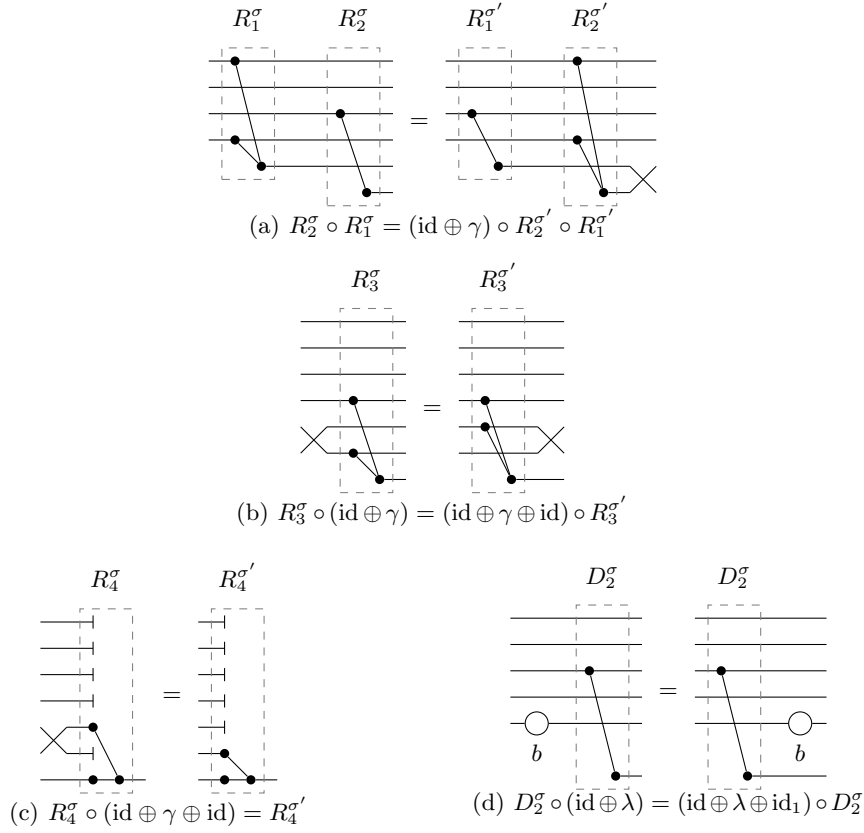
\begin{figure}[t]
\centering
\subfigure[$R_2^\sigma \comp R_1^\sigma = (\id \oplus \symm) \comp
R_2^{\sigma'} \comp R_1^{\sigma'}$]{
\begin{minipage}{6.5cm}\begin{center}
\begin{tikzpicture}[>=stealth,scale=0.7]

\draw (2.5cm,2.5cm) -- (6cm,2.5cm);
\draw (2.5cm,2cm) -- (6cm,2cm);
\draw (2.5cm,1.5cm) -- (6cm,1.5cm);
\draw (2.5cm,1cm) -- (6cm,1cm);

\draw (3cm,2.5cm)   node[fill,scale=0.4,shape=circle] (bin0) {};
\draw (3cm,1cm)     node[fill,scale=0.4,shape=circle] (bin1) {};
\draw (3.5cm,0.5cm) node[fill,scale=0.4,shape=circle] (bout) {};
\draw (bin0) -- (bout);
\draw (bin1) -- (bout);
\draw (bout) -- (6cm,0.5cm);
\draw[color=gray,dashed] (2.75cm,2.75) -- ++(1cm,0cm) -- ++(0cm,-2.5cm) -- ++(-1cm,0) -- cycle;
\draw (3.25cm,2.75cm) node[label=above:$R_1^\sigma$] () {};

\draw (5cm,1.5cm) node[fill,scale=0.4,shape=circle] (cin0) {};
\draw (5.5cm,0cm) node[fill,scale=0.4,shape=circle] (cout) {};
\draw (cin0) -- (cout);
\draw (cout) -- (6cm,0cm); 
\draw[color=gray,dashed] (4.75,2.75) -- ++(1cm,0cm) -- ++(0cm,-3cm) -- ++(-1cm,0) -- cycle;
\draw (5.25cm,2.75cm) node[label=above:$R_2^\sigma$] () {};
\end{tikzpicture}
\raisebox{30pt}{\ \ =\ }
\begin{tikzpicture}[>=stealth,scale=0.7]
\draw (2.5cm,2.5cm) -- (6.5cm,2.5cm);
\draw (2.5cm,2cm) -- (6.5cm,2cm);
\draw (2.5cm,1.5cm) -- (6.5cm,1.5cm);
\draw (2.5cm,1cm) -- (6.5cm,1cm);

\draw (3cm,1.5cm) node[fill,scale=0.4,shape=circle] (cin0) {};
\draw (3.5cm,0.5cm) node[fill,scale=0.4,shape=circle] (cout) {};
\draw (cin0) -- (cout);
\draw (cout) -- (6cm,0.5cm);
\draw[color=gray,dashed] (2.75cm,2.75) -- ++(1cm,0cm) -- ++(0cm,-2.5cm) -- ++(-1cm,0) -- cycle;
\draw (3.25cm,2.75cm) node[label=above:$R_1^{\sigma'}$] () {};

\draw (5cm,2.5cm)   node[fill,scale=0.4,shape=circle] (bin0) {};
\draw (5cm,1cm)     node[fill,scale=0.4,shape=circle] (bin1) {};
\draw (5.5cm,0cm) node[fill,scale=0.4,shape=circle] (bout) {};
\draw (bin0) -- (bout);
\draw (bin1) -- (bout);
\draw (bout) -- (6cm,0.0cm);
\draw[color=gray,dashed] (4.75,2.75) -- ++(1cm,0cm) -- ++(0cm,-3cm) -- ++(-1cm,0) -- cycle;
\draw (5.25cm,2.75cm) node[label=above:$R_2^{\sigma'}$] () {};

\draw (6cm,0.5cm) -- (6.5cm, 0cm);
\draw (6cm,0cm) -- (6.5cm, 0.5cm);
\end{tikzpicture}
\end{center}\end{minipage}
}
\qquad\qquad
\subfigure[$R_3^\sigma \comp (\id \oplus \symm)= (\id \oplus \symm\oplus \id) \comp R_3^{\sigma'}
  $]{
\begin{minipage}{5cm}\begin{center}
\begin{tikzpicture}[>=stealth,scale=0.7]
\draw (6cm,2.5cm) -- (8cm,2.5cm);
\draw (6cm,2cm) -- (8cm,2cm);
\draw (6cm,1.5cm) -- (8cm,1.5cm);
\draw (6cm,1cm) -- (8cm,1cm);
\draw (6cm,0.5cm) -- (6.5cm,0cm);
\draw (6cm,0cm) -- (6.5cm,0.5cm);
\draw (6.5cm,0.5cm) -- (8cm,0.5cm);
\draw (6.5cm,0cm) -- (8cm,0cm);

\draw (7cm,1cm) node[fill,scale=0.4,shape=circle] (din0) {};
\draw (7cm,0cm) node[fill,scale=0.4,shape=circle] (din1) {};
\draw (7.5cm,-0.5cm) node[fill,scale=0.4,shape=circle] (dout) {};
\draw (din0) -- (dout);
\draw (din1) -- (dout);
\draw (dout) -- (8cm,-0.5cm);
\draw[color=gray,dashed] (6.75,2.75) -- ++(1cm,0cm) -- ++(0cm,-3.5cm) -- ++(-1cm,0) -- cycle;
\draw (7.25cm,2.75cm) node[label=above:$R_3^\sigma$] () {};
\end{tikzpicture}
\raisebox{30pt}{\ \ =\ }
\begin{tikzpicture}[>=stealth,scale=0.7]
\draw (6.5cm,2.5cm) -- (8.5cm,2.5cm);
\draw (6.5cm,2cm) -- (8.5cm,2cm);
\draw (6.5cm,1.5cm) -- (8.5cm,1.5cm);
\draw (6.5cm,1cm) -- (8.5cm,1cm);
\draw (6.5cm,0.5cm) -- (8cm,0.5cm);
\draw (6.5cm,0cm) -- (8cm,0cm);
\draw (8cm,0.5cm) -- (8.5cm,0cm);
\draw (8cm,0cm) -- (8.5cm,0.5cm);

\draw (7cm,1cm) node[fill,scale=0.4,shape=circle] (din0) {};
\draw (7cm,0.5cm) node[fill,scale=0.4,shape=circle] (din1) {};
\draw (7.5cm,-0.5cm) node[fill,scale=0.4,shape=circle] (dout) {};
\draw (din0) -- (dout);
\draw (din1) -- (dout);
\draw (dout) -- (8.5cm,-0.5cm);
\draw[color=gray,dashed] (6.75,2.75) -- ++(1cm,0cm) -- ++(0cm,-3.5cm) -- ++(-1cm,0) -- cycle;
\draw (7.25cm,2.75cm) node[label=above:$R_3^{\sigma'}$] () {};
\end{tikzpicture}
\end{center}\end{minipage}
}
\\ 
\subfigure[$R_4^\sigma \comp (\id \oplus \symm \oplus \id)
  = R_4^{\sigma'}$]{
\begin{minipage}{4cm}\begin{center}
\begin{tikzpicture}[>=stealth,scale=0.7]
\draw (8cm, 2.5cm) -- (9cm,2.5cm);
\draw (9cm,2.375cm) -- (9cm, 2.625cm);
\draw (8cm, 2cm) -- (9cm,2cm);
\draw (9cm,1.875cm) -- (9cm, 2.125cm);
\draw (8cm, 1.5cm) -- (9cm,1.5cm);
\draw (9cm,1.375cm) -- (9cm, 1.625cm);
\draw (8cm, 1cm) -- (9cm,1cm);
\draw (9cm,0.875cm) -- (9cm, 1.125cm);

\draw (8cm, 0.5cm) -- (8.5cm,0cm);
\draw (8cm, 0cm) -- (8.5cm,0.5cm);

\draw (8.5cm, 0.5cm) -- (9cm,0.5cm);
\draw (8.5cm, 0cm) -- (9cm,0cm);
\draw (9cm,-0.125cm) -- (9cm, 0.125cm);

\draw (9cm,0.5cm) node[fill,scale=0.4,shape=circle] (outb) {};
\draw (9.5cm,-0.5cm) node[fill,scale=0.4,shape=circle] (out1) {};
\draw (9.0cm,-0.5cm) node[fill,scale=0.4,shape=circle] (outd) {};
\draw (outb) -- (out1);
\draw (8cm,-0.5cm) -- (out1);
\draw (out1) -- (10cm, -0.5cm);

\draw[color=gray,dashed] (8.75cm,2.75cm) -- ++(1cm,0cm) -- ++(0cm,-3.5cm) -- ++(-1cm,0) -- cycle;
\draw (9.25cm,2.75cm) node[label=above:$R_4^\sigma$] () {};
\end{tikzpicture}
\raisebox{35pt}{\ \ =\ }
\begin{tikzpicture}[>=stealth,scale=0.7]
\draw (8.5cm, 2.5cm) -- (9cm,2.5cm);
\draw (9cm, 2.375cm) -- (9cm, 2.625cm);
\draw (8.5cm, 2cm) -- (9cm,2cm);
\draw (9cm, 1.875cm) -- (9cm, 2.125cm);
\draw (8.5cm, 1.5cm) -- (9cm,1.5cm);
\draw (9cm, 1.375cm) -- (9cm, 1.625cm);
\draw (8.5cm, 1cm) -- (9cm,1cm);
\draw (9cm, 0.875cm) -- (9cm, 1.125cm);
\draw (8.5cm, 0.5cm) -- (9cm,0.5cm);
\draw (9cm, 0.625cm) -- (9cm, 0.375cm);
\draw (8.5cm, 0cm) -- (9cm,0cm);

\draw (9cm,0cm) node[fill,scale=0.4,shape=circle] (outb) {};
\draw (9.5cm,-0.5cm) node[fill,scale=0.4,shape=circle] (out1) {};
\draw (outb) -- (out1);
\draw (8.5cm,-0.5cm) -- (out1);
\draw (out1) -- (10cm, -0.5cm);
\draw (9.0cm,-0.5cm) node[fill,scale=0.4,shape=circle] (outd) {};

\draw[color=gray,dashed] (8.75cm,2.75cm) -- ++(1cm,0cm) -- ++(0cm,-3.5cm) -- ++(-1cm,0) -- cycle;
\draw (9.25cm,2.75cm) node[label=above:$R_4^{\sigma'}$] () {};
\end{tikzpicture}
\end{center}\end{minipage}
}
\qquad\qquad\qquad
\subfigure[$D_2^\sigma \comp (\id \oplus \lambda) 
  = (\id \oplus \lambda \oplus \id_1) \comp D_2^\sigma$]{ 
\begin{minipage}{5cm}\begin{center}
\begin{tikzpicture}[>=stealth,scale=0.7]
\draw (4cm,0.5cm)   node[draw,shape=circle,label=below:$b$] (b) {};
\draw (3.5cm,2.5cm) -- ++(2.5cm,0cm);
\draw (3.5cm,2cm) -- ++(2.5cm,0cm);
\draw (3.5cm,1.5cm) -- ++(2.5cm,0cm);
\draw (3.5cm,1cm) -- ++(2.5cm,0cm);
\draw (3.5cm,0.5cm) -- (b);
\draw (b) -- (6cm,0.5cm);

\draw (5cm,1.5cm) node[fill,scale=0.4,shape=circle] (cin0) {};
\draw (5.5cm,-.5cm) node[fill,scale=0.4,shape=circle] (cout) {};
\draw (cin0) -- (cout);
\draw (cout) -- (6cm,-.5cm);
\draw[color=gray,dashed] (4.75,2.75) -- ++(1cm,0cm) -- ++(0cm,-3.5cm) -- ++(-1cm,0) -- cycle;
\draw (5.25cm,2.75cm) node[label=above:$D_2^\sigma$] () {};
\end{tikzpicture}
\raisebox{35pt}{\ \ =\ }
\begin{tikzpicture}[>=stealth,scale=0.7]
\draw (6.5cm,0.5cm)   node[draw,shape=circle,label=below:$b$] (b) {};
\draw (4.5cm,2.5cm) -- ++(2.5cm,0cm);
\draw (4.5cm,2cm) -- ++(2.5cm,0cm);
\draw (4.5cm,1.5cm) -- ++(2.5cm,0cm);
\draw (4.5cm,1cm) -- ++(2.5cm,0cm);
\draw (4.5cm,0.5cm) -- (b);
\draw (b) -- (7cm,0.5cm);

\draw (5cm,1.5cm) node[fill,scale=0.4,shape=circle] (cin0) {};
\draw (5.5cm,-.5cm) node[fill,scale=0.4,shape=circle] (cout) {};
\draw (cin0) -- (cout);
\draw (cout) -- (7cm,-.5cm);
\draw[color=gray,dashed] (4.75,2.75) -- ++(1cm,0cm) -- ++(0cm,-3.5cm) -- ++(-1cm,0) -- cycle;
\draw (5.25cm,2.75cm) node[label=above:$D_2^\sigma$] () {};
\end{tikzpicture}
\end{center}\end{minipage}
}
\caption{Graphical demonstration of the identities in
  Lemma~\ref{lemma:transposition_props} for the (\ord3,\ord1)-dag of
  Fig.~\ref{Adag}.} 
\label{fig:transposition_props}
\end{figure}

More fundamental is the independence of the interpretation under
topological sorting.

\begin{lemma}
\label{lemma:transposition_props}
Let $\sigma$ and $\sigma'$ be two topological sortings of a finite
\mbox{$(\ord n,\ord m)$-dag} $D=(N,R)$ with $|N|\geq 2$.  If $\sigma$ and
$\sigma'$ differ only by the transposition of two adjacent indices, say
$\sigma'_i=\sigma_{i+1}$ and $\sigma'_{i+1}=\sigma_i$ for $0\leq i<|N|-1$,
then the following identities hold:
\begin{enumerate}
\item 
\label{lemma:transposition_props_ONE}
  $R_j^\sigma = R_j^{\sigma'}$ for all $0\leq j<i$, 
  \smallskip

\item 
\label{lemma:transposition_props_TWO}
  $R_{i+1}^\sigma \comp R_i^\sigma 
   = 
   (\id_{n+i} \oplus \symm) \comp R_{i+1}^{\sigma'} \comp
   R_{i}^{\sigma'}$,
  \smallskip

\item
\label{lemma:transposition_props_THREE}
  $R_j^\sigma \comp (\id_{n+i} \oplus \symm\oplus \id_{j-i-2})
   = (\id_{n+i} \oplus \symm \oplus \id_{j-i-1}) \comp R_j^{\sigma'}$ 
  for all $i+1<j<|N|$,\\[-2mm]

\item
\label{lemma:transposition_props_FOUR}
  $R_{|N|}^{\sigma} \comp (\id_{n+i} \oplus \symm \oplus \id_{|N|-i-2})
   = R_{|N|}^{\sigma'}$,
  \smallskip

\item
\label{lemma:transposition_props_FIVE}
 $\semR{R_{i+1}^\tau} \comp (\id_{A^{\tensor n+i}} \tensor \lambda_A) 
  = (\id_{A^{\tensor n+i}} \tensor \lambda_A \tensor \id_A) \comp
  \semR{R_{i+1}^\tau}$ for $\tau=\sigma,\sigma'$. 
\end{enumerate}
\end{lemma}
\hide{
\begin{proof}[Proof notes]
(\ref{lemma:transposition_props_TWO}) Intuitively, references to $i$ and
$i+1$ are swapped in $R_j^\sigma$ and $R_j^{\sigma'}$.  
\begin{itemize} 
  \item 
    $R_{i+1}^\sigma \comp R_i^\sigma = \emptyset$.
    \smallskip
  
  \item 
    $R_i^\sigma
     = (\id_{n+i} \oplus \epsilon \oplus \id_1) \comp R_{i+1}^{\sigma'}
       \comp (\id_{n+i} \oplus \eta)$.
    \smallskip
  
  \item 
    $R_{i+1}^\sigma 
       = (\id_{n+i} \oplus \symm) \comp 
         \big(R_i^{\sigma'} \oplus (\eta\comp\epsilon)\big)$.
    \smallskip

   \item 
     for $\rho:n+\ell \rightarrow n+\ell$ in $\DCBiAlgPROP$,
     $(\rho \oplus \id_1) \comp R_\ell^\tau = R_\ell^\tau$.
\end{itemize}
\end{itemize}
\end{proof}
}
\begin{proof}
The
identities~(\ref{lemma:transposition_props_ONE}--\ref{lemma:transposition_props_FOUR})
follow by construction.  
Identity~(\ref{lemma:transposition_props_FIVE}) is a consequence of the
following general fact: for every $R\in\DCBiAlgPROP(\ord{k+1},\ord{k+2})$ such
that, for all ${j\in\ord{k+2}}$, $(k,j)\in R$ iff $j=k$ one has
$R=(\id_k\oplus\gamma)\comp(R'\oplus\id_1)$ for 
${R'\in\DCBiAlgPROP(\ord k,\ord{k+1})}$; so that, for all $f:A\to A$,
\[\begin{array}{rcl}
\semR{R}_A\comp(\id_{A^{\tensor k}}\tensor f)
& = & 
(\id_{A^{\tensor k}}\tensor\gamma) \comp 
(\semR{R'}_A\tensor\id_A) \comp
(\id_{A^{\tensor k}}\tensor f) 
\\[2mm]
& = & 
(\id_{A^{\tensor k}}\tensor\gamma) \comp 
(\id_{A^{\tensor k+1}}\tensor f) \comp
(\semR{R'}_A\tensor\id_A)
\\[2mm]
& = & 
(\id_{A^{\tensor k}}\tensor f\tensor\id_A) \comp
(\id_{A^{\tensor k}}\tensor\gamma) \comp 
(\semR{R'}_A\tensor\id_A)
\\[2mm]
& = & 
(\id_{A^{\tensor k}}\tensor f\tensor\id_A) \comp
\semR{R}_A
\end{array}\]
\end{proof}

\begin{proposition}\label{TopSortInvariance}
For any two topological sortings $\sigma,\sigma'$ of a finite 
\mbox{$(\ord n,\ord m)$-dag}~$D$, 
\[
  \psemD D\sigma_A = \psemD D{\sigma'}_A
  : A^{\tensor n}\to A^{\tensor m}
  \enspace.
\]
\end{proposition}
\begin{proof}
It is enough to establish the equality for $\sigma$ and $\sigma'$ as in
the hypothesis of Lemma~\ref{lemma:transposition_props}.  Let us then
assume this situation.

By
Lemma~\ref{lemma:transposition_props}~(\ref{lemma:transposition_props_ONE}),
we have
\[\begin{array}{ll}
(\id\tensor\lambda_A) \comp \semR{R_{i-1}^\sigma}_A \comp \cdots \comp
(\id\tensor\lambda_A) \comp \semR{R_{0}^\sigma}_A
\\[2mm]
\quad = \
(\id\tensor\lambda_A) \comp \semR{R_{i-1}^{\sigma'}}_A \comp \cdots \comp
(\id\tensor\lambda_A) \comp \semR{R_{0}^{\sigma'}}_A
\end{array}\]
so that we need only show
\[\begin{array}{l}
\semR{R_{|N|}^\sigma}_A \comp (\id\tensor\lambda_A) \comp
\semR{R_{|N|-1}^\sigma}_A \comp \cdots \comp (\id\tensor\lambda_A) \comp
\semR{R_{i}^\sigma}_A
\\[2mm]
\quad = \
\semR{R_{|N|}^{\sigma'}}_A \comp (\id\tensor\lambda_A) \comp
\semR{R_{|N|-1}^{\sigma'}}_A \comp \cdots \comp (\id\tensor\lambda_A) \comp
\semR{R_{i}^{\sigma'}}_A
\end{array}\]
For this we calculate in three steps as follows:
\begin{enumerate}
\item
$\begin{array}[t]{ll}
  (\id\tensor\lambda_A) \comp \semR{R_{i+1}^\sigma}_A \comp
  (\id\tensor\lambda_A) \comp \semR{R_{i}^\sigma}_A
  \\[2mm]
  \quad =\
  (\id\tensor\lambda_A) \comp (\id\tensor\lambda_A\tensor\id_A) \comp
  \semR{R_{i+1}^\sigma}_A \comp \semR{R_{i}^\sigma}_A
  \\[1mm]
  \qquad\qquad \text{, by
  Lemma~\ref{lemma:transposition_props}~(\ref{lemma:transposition_props_FIVE})}
  \\[2mm]
  \quad =\
  (\id\tensor\lambda_A) \comp (\id\tensor\lambda_A\tensor\id_A) \comp
  (\id\tensor\symm) \comp \semR{R_{i+1}^{\sigma'}}_A \comp
  \semR{R_{i}^{\sigma'}}_A
  \\[1mm]
  \qquad\qquad\text{, by
  Lemma~\ref{lemma:transposition_props}~(\ref{lemma:transposition_props_TWO})}
  \\[2mm]
  \quad =\
  (\id\tensor\symm) \comp (\id\tensor\lambda_A) \comp
  (\id\tensor\lambda_A\tensor\id_A) \comp \semR{R_{i+1}^{\sigma'}}_A
  \comp \semR{R_{i}^{\sigma'}}_A
  \\[2mm]
  \quad =\
  (\id\tensor\symm) \comp (\id\tensor\lambda_A) \comp
  \semR{R_{i+1}^{\sigma'}}_A \comp (\id\tensor\lambda_A) \comp
  \semR{R_{i}^{\sigma'}}_A
  \\[1mm]
  \qquad\qquad\text{, by
  Lemma~\ref{lemma:transposition_props}~(\ref{lemma:transposition_props_FIVE})}
\end{array}$

\item
$\begin{array}[t]{l}
    (\id_{A^{\tensor n+|N|-1}}\tensor\lambda_A) \comp
    \semR{R_{|N|-1}^\sigma}_A \comp \cdots 
    \\[2mm]
    \quad
    \cdots \comp 
    (\id_{A^{\tensor n+i+2}}\tensor\lambda_A) \comp
    \semR{R_{i+2}^\sigma}_A \comp (\id_{A^{\tensor n+i}}\tensor\symm)
    \\[2mm]
    \quad = \
    (\id_{A^{\tensor n+|N|-1}}\tensor\lambda_A) \comp
    \semR{R_{|N|-1}^\sigma}_A \comp 
    \cdots 
    \\[2mm]
    \qquad\qquad
    \cdots \comp
    (\id_{A^{\tensor n+i+2}}\tensor\lambda_A) \comp
    (\id_{A^{\tensor n+i}}\tensor\symm\tensor\id_A) \comp
    \semR{R_{i+2}^{\sigma'}}_A
    \\[2mm]
    \qquad\qquad
  \text{, by
  Lemma~\ref{lemma:transposition_props}~(\ref{lemma:transposition_props_THREE})}
    \\[2mm]
    \quad = \
    (\id_{A^{\tensor n+|N|-1}}\tensor\lambda_A) \comp
    \semR{R_{|N|-1}^\sigma}_A \comp 
    \cdots 
    \\[2mm]
    \qquad\qquad
    \cdots \comp
    (\id_{A^{\tensor n+i}}\tensor\symm\tensor\id_A) 
    \comp
    (\id_{A^{\tensor n+i+2}}\tensor\lambda_A) 
    \comp
    \semR{R_{i+2}^{\sigma'}}_A
    \\
    \quad\ \vdots
    \\
    \quad = \
    (\id_{A^{\tensor n+i}}\tensor\symm\tensor\id_{A^{\tensor|N|-i-2}})
    \comp (\id_{A^{\tensor n+|N|-1}}\tensor\lambda_A) \comp
    \semR{R_{|N|-1}^{\sigma'}}_A \comp 
    \cdots 
    \\[2mm]
    \qquad\qquad
    \cdots \comp 
    (\id_{A^{\tensor n+i+2}}\tensor\lambda_A) \comp
    \semR{R_{i+2}^{\sigma'}}_A
  \end{array}$

\item
  $\begin{array}[t]{ll}
    \semR{R_{|N|}^\sigma}_A
    \comp
    (\id_{A^{\tensor n+i}}\tensor\symm\tensor\id_{A^{\tensor|N|-i-2}})
    =
    \semR{R_{|N|}^{\sigma'}}_A
    &
\text{, by
  Lemma~\ref{lemma:transposition_props}~(\ref{lemma:transposition_props_FOUR})}
\end{array}$
\end{enumerate}
\end{proof}

\myparagraph{Compositionality} 

We show that the interpretation of finite idags is compositional for the
operations of concatenation and juxtaposition.

\begin{proposition}
Let $D=(N,R)$ be a finite $(\ord n,\ord m)$-dag topologically sorted by
$\sigma$ and $D'=(N',R')$ a finite $(\ord m,\ord \ell)$-dag topologically
sorted by $\sigma'$.  Write $\sigma'\after\sigma$ for the topological
sorting of the concatenation $(\ord n,\ord\ell)$-dag $D'\concat
D=(N+N',R'')$ according to $\sigma$ and then $\sigma'$ (that is, with
$(\sigma'\after\sigma)_i = \sigma_i$ for $0\leq i< |N|$ and
$(\sigma'\after\sigma)_{|N|+j} = \sigma'_{j}$ for $0\leq j<|N'|$).  Then,
\[
  \psemD{D'}{\sigma'}_A \comp \psemD{D}{\sigma}_A
  =
  \psemD{D'\concat D}{\sigma'\after\sigma}_A
  \enspace.
\]
\end{proposition}
\begin{proof}
The result follows from the definition of the interpretation function and the
following identities:
\begin{enumerate}
\item
  ${R}_i^\sigma = {R''}_i^{(\sigma'\after\sigma)}$ for all $0\leq i<|N|$,
  \smallskip

\hide{
\item
\[
\ol R^{\sigma'\after\sigma}_{|N|+j} 
= ???
(\ol R^{\sigma}_{|N|} \oplus \id) \comp (\ol
R^{\sigma'\after\sigma}_{|N|+j} \oplus \id) 
\] 
}

\item
  ${R'}_j^{\sigma'} \comp \big(R_{|N|}^\sigma\oplus \id_j\big)
   =
   \big(R_{|N|}^\sigma\oplus\id_{j+1}\big)
   \comp {R''}_{|N|+j}^{(\sigma'\after\sigma)}$
  for all $0\leq j<|N'|$,
  \smallskip

\item
  ${R'}_{|N'|}^{\sigma'}\comp \big( R_{|N|}^\sigma \oplus \id_{|N'|}\big)
   =
   {R''}_{|N+N'|}^{(\sigma'\after\sigma)}$.
\end{enumerate}
\end{proof}

\begin{proposition}
%
Let $D=(N,R)$ be a finite \mbox{$(\ord n,\ord m)$-dag} topologically sorted by
$\sigma$ and let $D'=(N',R')$ be a finite \mbox{$(\ord{n'},\ord{m'})$-dag}
topologically sorted by $\sigma'$.  
The \mbox{$(\ord{n+n'},\ord{m+m'})$-dag}~$D\juxt D'=(N+N',R'')$ obtained by
juxtaposition is topologically sorted by
$\sigma'\after\sigma$ and 
\[
  \psemD{D}{\sigma}_A \tensor \psemD{D'}{\sigma'}_A
  =
  \psemD{D\juxt D'}{\sigma'\after\sigma}_A 
  \enspace.
\]
\end{proposition}
\begin{proof}
The result follows from the definition of the interpretation function and the
following identities:
\begin{enumerate}
  \item
    ${R''}_i^{(\sigma'\after\sigma)} = R_i^\sigma$
    for all $0\leq i<|N|$,
    \smallskip

  \item
    ${R''}_{|N|+j}^{(\sigma'\after\sigma)} 
     = 
     \id_{n+|N|} \oplus {R'}_j^{\sigma'}$ 
     for all $0\leq j<|N'|$,
    \smallskip

  \item
    ${R''}_{|N+N'|}^{(\sigma'\after\sigma)}
     = {R}_{|N|}^{\sigma} \oplus {R'}_{|N'|}^{\sigma'}$.
\end{enumerate}
\end{proof}

\begin{proof}[of Lemma~\ref{PushoutLemma}]
For a cone 
\[\xymatrix@C15pt@R15pt{
    & \ar[dl] \ar[dd]|-G \EmptyPROP \ar[dr] & 
    \\
    \DCBiAlgPROP \ar[dr]_-F & & \ar[dl]^-H \LoopPROP
    \\
    & \cat C &
}\]
of symmetric monoidal categories, consider the $\DCBiAlgLoopTheory$-algebra
\[\begin{array}{ll}
  \multicolumn{2}{c}{
  A = G1
  }
  \\
  \eta_A = (I\iso F0\xymatrix{\ar[r]^-{F\eta}&} A)
  & ,\quad
  \nabla_A = (A\tensor A\iso F(2)\xymatrix{\ar[r]^-{F\nabla}&} A)
  \\
  \epsilon_A = (A\xymatrix{\ar[r]^-{F\epsilon}&} F0\iso I)
  & ,\quad
  \Delta_A = (A \xymatrix{\ar[r]^-{F\Delta}&} F(2)\iso A\tensor A )
  \\
  \multicolumn{2}{c}{
  \lambda_A = (A \xymatrix{\ar[r]^-{H\lambda}&} A)
  }
\end{array}\]
and define the unique mediating functor~$\DagPROP\to\cat C$ to map
$D\in\DagPROP(\ord n,\ord m)$ to the composite
\[
  \semD{D}_A
  =
  \big(\, G(n)\iso A^{\tensor n}
  \xymatrix@C60pt{\ar[r]^-{\psemD{(N,R)}\sigma_A}&} A^{\tensor m} \iso
  G(m) \,\big)
\]
for a topological sorting $\sigma$ of a representation $(N,R)$ of the
abstract idag~$D$.  (The symmetric monoidal structure of this functor is
inherited from that of $G$.)
\end{proof}

\hide{
We have thus established that the {\PROP} of finite abstract idags is free on
the symmetric monoidal equational theory of degenerate commutative bialgebras
with a node.
}

\mysection{Conclusion}

We have given an algebraic presentation of dag structure in the categorical
language of {\PROPs}, establishing that the {\PROP} of finite abstract
interfaced dags is universally characterised as being free on the symmetric
monoidal equational theory of degenerate commutative bialgebras with a node.
A main contribution in this respect has been a simple proof that provides an
initial-algebra semantics for dag structure.

The technique introduced in the paper is robust and can be adapted to a
variety of similar results.  
Firstly, one may drop the degeneracy condition on bialgebras.  In this case,
the free {\PROP} on the sum of the symmetric monoidal equational theories
$\BiAlgTheory$ and $\LoopTheory$ consists of idags with edges weighted by
positive natural
numbers.  These can be formalised as structures
$\big(I,O,N,R\in\Nat^{(I+N)\times(O+N)}\big)$ such that $\big(I,O,N,\setof{
  (x,y) \suchthat R(x,y)\not= 0 }\big)$ is an idag.
Secondly, one may introduce an antipode operator.  In this case, the free
{\PROP} on the sum of the symmetric monoidal equational theories
$\HopfAlgTheory$ and $\LoopTheory$ consists of idags with edges weighted by
non-zero integers.  Analogously, these can be formalised as structures
$\big(I,O,N,{R\in\Int^{(I+N)\times(O+N)}}\big)$ such that
$\big(I,O,N,\setof{ (x,y) \suchthat R(x,y)\not= 0 }\big)$ is an idag.
Of course, these two weightings respectively come from the structure of
$\BiAlgPROP$ and $\HopfAlgPROP$, see
\S\S~\ref{CommutativeBialgebras} and~\ref{CommutativeHopfAlgebras}.
Finally, one may generalise from $\LoopTheory$ to $\LoopsTheory L$ for a set
of labels~$L$.  The resulting free {\PROPs} consist of the appropriate
versions of \mbox{$L$-labelled} idags.

In another direction, one may consider extending the symmetric monoidal
theory~$\DCBiAlgLoopTheory$ with equations involving the node.  As
suggested to us by Samuel Mimram, an interesting possibility is to
introduce the equation 
\[
  \lambda = \nabla \comp (\lambda\oplus\id_1) \comp \Delta
  \ : 1\to1
  \enspace.
\]
According to the canonical decomposition~(\ref{DagDecomposition}), the
effect of this equation on the free {\PROP}~$\DagPROP$ is to force on 
idags~$D$ the identification 
\begin{align*}
  D  
  & =
  D_{|N|}^\sigma \comp (\id_{n+|N|-1} \oplus \lambda) \comp
  D_{|N|-1}^\sigma \comp \cdots \comp (\id_{n} \oplus \lambda) \comp
  D_0^\sigma
  \\[2mm]
  & = 
  D_{|N|}^\sigma \comp 
  \big(\id_{n+|N|-1} \oplus (\nabla\comp(\lambda\oplus\id_1)\comp\Delta)\big) 
  \comp D_{|N|-1}^\sigma
  \comp \cdots
  \\[1mm]
  &
  \qquad\qquad
  \cdots \comp 
  \big(\id_{n} \oplus (\nabla\comp(\lambda\oplus\id_1)\comp\Delta)\big) \comp
  D_0^\sigma 
  \\[2mm]
  & = 
  D^+
\end{align*}
for $D^+$ the transitive closure of $D$.  The free
{\PROP} 
consists then 
of transitive 
idags. 
For another example, one may consider introducing the equations
\[
  \lambda\comp\eta = \eta
  \ : 0\to 1
  \enspace,\quad
  \epsilon\comp\lambda = \epsilon
  \ : 1\to 0
  \enspace.
\]
The resulting free {\PROP} is that of idags with no dangling internal
nodes.



\end{document}